\documentclass[leqno,12pt]{article} 
\usepackage{amsthm}
\usepackage{amscd}
\usepackage{amsthm}
\usepackage{graphicx}
\usepackage{amsmath}
\usepackage{amsfonts}
\usepackage{amssymb}

\newcommand\supp{\mathop{\rm supp}}

\theoremstyle{plain} 
\newtheorem{theorem}{\indent\sc Theorem}[section]
\newtheorem{lemma}[theorem]{\indent\sc Lemma}
\newtheorem{corollary}[theorem]{\indent\sc Corollary}
\newtheorem{proposition}[theorem]{\indent\sc Proposition}

\theoremstyle{definition} 
\newtheorem{definition}[theorem]{\indent\sc Definition}
\newtheorem{remark}[theorem]{\indent\sc Remark}

\makeatletter
\def\address#1#2{\begingroup
\noindent\parbox[t]{7.8cm}{%
\small{\scshape\ignorespaces#1}\par\vskip1ex
\noindent\small{\itshape E-mail address}%
\/: #2\par\vskip4ex}\hfill%
\endgroup}%
\makeatother
\title{A molecular decomposition for $H^p(\mathbb{Z}^n)$} 
\author{
\textsc{Pablo Rocha} 
}
\date{} 
\begin{document}

\maketitle


\begin{abstract}
In this work, for the range $\frac{n-1}{n} < p \leq 1$, we give a molecular reconstruction theorem for $H^p(\mathbb{Z}^n)$. As an application of this result and the atomic decomposition developed by S. Boza and M. Carro in \cite{Carro}, we prove that the discrete 
Riesz potential $I_{\alpha}$ defined on $\mathbb{Z}^n$ is a bounded operator $H^p(\mathbb{Z}^n) \to H^q(\mathbb{Z}^n)$ for 
$\frac{n-1}{n} < p < \frac{n}{\alpha}$ and $\frac{1}{q} = \frac{1}{p} - \frac{\alpha}{n}$, where $0 < \alpha < n$. 
\end{abstract}

\section{Introduction}

S. Boza and M. Carro in \cite{Carro} (see also \cite{Boza}) introduced the discrete Hardy spaces on $\mathbb{Z}^n$. They gave a variety of distinct approaches, based on differing definitions, all leading to the same notion of Hardy spaces $H^p(\mathbb{Z}^n)$. We briefly recall these characterizations. Consider the discrete Poisson kernel on $\mathbb{Z}^n$, which is defined by
\[
P_t^d(j) = C_n \frac{t}{(t^2 + |j|^2)^{(n+1)/2}}, \,\,\,\, t > 0, \,\, j \in \mathbb{Z}^n \setminus \{ {\bf 0} \}, \,\, P_t^d({\bf 0}) = 0,
\]
where $C_n$ is a normalized constant depending on the dimension. For a sequence $b= \{ b(j) \}_{j \in \mathbb{Z}^n}$, let
\[
\| b \|_{\ell^{p}(\mathbb{Z})} = \left\{\begin{array}{cc}
                  \left( \displaystyle{\sum_{j \in \mathbb{Z}^n}} |b(j)|^{p} \right)^{1/p}, &  0 < p < \infty \\
                  \displaystyle{\sup_{j \in \mathbb{Z}^n}} \, |b(j)| \, , & \,\,\,\, p = \infty
                \end{array}. \right.
\]
A sequence $b= \{ b(j) \}_{j \in \mathbb{Z}^n}$ is said to belong to $\ell^{p}(\mathbb{Z}^n)$, $0 < p \leq \infty$, if 
$\| b \|_{\ell^{p}(\mathbb{Z}^n)} < \infty$.

Then, for $0 < p \leq 1$, we define
\[
H^p(\mathbb{Z}^n) = \left\{ b \in \ell^p(\mathbb{Z}^n) : \sup_{t>0} |(P_t^d \ast_{\mathbb{Z}^n} b)| \in \ell^p(\mathbb{Z}^n) \right\},
\]
with the "$p$-norm" given by
\[
\| b \|_{H^p(\mathbb{Z}^n)} := \| b \|_{\ell^p(\mathbb{Z}^n)} + \| (P_t^d \ast_{\mathbb{Z}^n} b) \|_{\ell^p(\mathbb{Z}^n)}.
\]
By \cite[Theorem 2.7]{Carro} the discrete Poisson kernel $P_t^d$ can be substituted by $\Phi_t^d$, where 
$\Phi \in \mathcal{S}(\mathbb{R}^n)$, $\int_{\mathbb{R}^n} \Phi =1$, $\Phi_t^d(j) = t^{-n} \Phi(j/t)$ if $j \neq {\bf 0}$ and 
$\Phi_t^d({\bf 0}) = 0$. Moreover, the respective $H^p$-norms are equivalent.

Now, for $s=1, ..., n$, we introduce the discrete Riesz kernels $K_{s}^d$ on $\mathbb{Z}^n$, that is
\[
K_{s}^d(j) = \frac{j_s}{|j|^{n+1}}, \,\,\,\, \text{for} \,\, s=1, ..., n \,\, \text{and} \,\, j = (j_1, ..., j_n) \in 
\mathbb{Z}^n \setminus \{ {\bf 0 } \},
\]
and $K_{s}^d({\bf 0 }) = 0$. The discrete Riesz transforms, $R_{s}^d$, applied to a sequence $b = \{ b(j) \}_{j \in \mathbb{Z}^n}$ are the
convolution operators
\[
(R_{s}^d b)(m) = (K_{s}^d \ast_{\mathbb{Z}^n} b)(m) = \sum_{j \neq m} b(j) \frac{m_s - j_s}{|m-j|^{n+1}}, \,\,\,\, s = 1, ..., n.
\]
Then, for $0 < p < \infty$, one defines
\[
H^p_{\text{Riesz}}(\mathbb{Z}^n) = \left\{ b \in \ell^p(\mathbb{Z}^n) : R_{s}^d b \in \ell^p(\mathbb{Z}^n), \, s=1, ..., n \right\},
\]
with the "$p$-norm" given by
\begin{equation} \label{Riesz norm}
\| b \|_{H^p_{\text{Riesz}}(\mathbb{Z}^n)} := \| b \|_{\ell^p(\mathbb{Z}^n)} + \sum_{s=1}^n \| R_{s}^d b \|_{\ell^p(\mathbb{Z}^n)}.
\end{equation}
For $\frac{n-1}{n} < p \leq 1$, \cite[Theorem 2.6]{Carro} states that $H^p(\mathbb{Z}^n) = H^p_{\text{Riesz}}(\mathbb{Z}^n)$, 
with equivalent $H^p$-norms. For $1 < p < \infty$, we define $H^p(\mathbb{Z}^n) := H^p_{\text{Riesz}}(\mathbb{Z}^n) = \ell^p(\mathbb{Z}^n)$ 
(this last identity follows from \cite[Proposition 12]{Auscher}) and put $H^{\infty}(\mathbb{Z}^n) := \ell^{\infty}(\mathbb{Z}^n)$. In 
\cite{Boza2}, S. Boza and M. Carro established the connection between the boundedness of convolution operators on $H^p(\mathbb{R}^n)$ and some related operators on $H^p(\mathbb{Z}^n)$.

In \cite{Carro}, the authors also gave an atomic characterization of $H^p(\mathbb{Z}^n)$ for $0 < p \leq 1$. Before establishing this result we give the definition of $(p, \infty, L)$-atom in $H^p(\mathbb{Z}^n)$. 
\begin{definition}
Let $0 < p \leq 1 < p_0 \leq \infty$, $d_p := \lfloor n(p^{-1} - 1) \rfloor$ 
($\lfloor s \rfloor$ indicates the integer part of $s \geq 0$) and $L \geq d_p$, with $L \in \mathbb{N}_0 := \mathbb{N} \cup \{ 0 \}$. We say that a sequence $a = \{ a(j) \}_{j \in \mathbb{Z}^n}$ is an $(p, p_0, L)$-atom centered at a discrete cube $Q \subset \mathbb{Z}^n$ if the following three conditions hold:

(a1) $\supp a \subset Q$,

(a2) $\| a \|_{\ell^{p_0}(\mathbb{Z}^n)} \leq (\# Q)^{1/p_0 - 1/p}$, where $\# Q$ represents the cardinality of $Q$,

(a3) $\displaystyle{\sum_{j \in Q}} j^{\beta} a(j) = 0$, for all multi-index $\beta=(\beta_1, ..., \beta_n) \in \mathbb{N}_0^n$ with 
$[\beta]:=\beta_1 + \cdot \cdot \cdot + \beta_n \leq L$, where $j^{\beta} := j_{1}^{\beta_1} \cdot \cdot \cdot j_{n}^{\beta_n}$.
\end{definition}

\begin{remark} \label{infinity atom}
It is easy to check that every $(p, \infty, L)$-atom is an $(p, p_0, L)$-atom for each $1 < p_0 < \infty$.
\end{remark}

The atomic decomposition for members in $H^p(\mathbb{Z}^n)$, $0 < p \leq 1$, developed in \cite{Carro} is as follows:

\begin{theorem} (\cite[Theorem 3.7]{Carro}) \label{atom decomp} Let $0 < p \leq 1$, $L \geq d_p$ and $b \in H^{p}(\mathbb{Z}^n)$. 
Then there exist a sequence of $(p, \infty, L)$-atoms $\{ a_k \}_{k=0}^{+\infty}$, a sequence of scalars 
$\{ \lambda_k \}_{k=0}^{+\infty}$ and a positive constant $C$, which depends only on $p$ and $n$, with 
$\sum_{k=0}^{+\infty} |\lambda_k |^{p} \leq C \| b \|_{H^{p}(\mathbb{Z}^n)}^{p}$ such that $b = \sum_{k=0}^{+\infty} \lambda_k a_k$, 
where the series converges in $H^{p}(\mathbb{Z}^n)$.
\end{theorem}

The purpose of this article is to continue the study about the behavior of discrete Riesz potential on discrete Hardy spaces 
$H^{p}(\mathbb{Z}^n)$, with $n \geq 1$, began by the author in \cite{Rocha2}. Therein, for $0 < \alpha < n$ and a sequence 
$b = \{ b(i) \}_{i \in \mathbb{Z}^n}$, we consider the discrete Riesz potential defined by
\[
(I_{\alpha}b)(j) = \sum_{i \in \mathbb{Z}^n \setminus \{ j \}} \frac{b(i)}{|i-j |^{n - \alpha}}, \,\,\,\,\,\, j \in \mathbb{Z}^n.
\]
By means of Theorem \ref{atom decomp} and the boundedness of discrete fractional maximal operator, we proved in \cite{Rocha2} that 
$I_{\alpha}$ is a bounded operator $H^{p}(\mathbb{Z}^n) \to \ell^{q}(\mathbb{Z}^n)$ for $0 < p < \frac{n}{\alpha}$ and 
$\frac{1}{q} = \frac{1}{p} - \frac{\alpha}{n}$. Y. Kanjin and M. Satake in \cite{Kanjin} studied the discrete Riesz potential $I_{\alpha}$ for the case $n=1$ and proved the $H^p(\mathbb{Z}) \to H^q(\mathbb{Z})$ boundedness of $I_{\alpha}$, for $0 < p < \alpha^{-1}$ and 
$\frac{1}{q} = \frac{1}{p} - \alpha$. To achieve this result, they furnished a molecular decomposition for $H^p(\mathbb{Z})$ analogous to the ones given by M. Taibleson and G. Weiss in \cite{Taible} for the Hardy spaces $H^p(\mathbb{R}^n)$.

In this work, we prove that $I_{\alpha}$ is a bounded operator $H^{p}(\mathbb{Z}^n) \to H^{q}(\mathbb{Z}^n)$ for 
$\frac{n-1}{n} < p < \frac{n}{\alpha}$ and $\frac{1}{q} = \frac{1}{p} - \frac{\alpha}{n}$. For them, as in \cite{Kanjin}, we furnish a molecular decomposition for the elements of $H^{p}(\mathbb{Z}^n)$, in the range $\frac{n-1}{n} < p \leq 1$. 

For more results about discrete fractional type operators one can consult \cite{Hardy}, \cite{Wainger}, \cite{Oberlin} and \cite{Rocha}.

The paper is organized as follows. Section 2 begins with the preliminaries. In Section 3, we present the concept of molecule in 
$\mathbb{Z}^n$ and prove some of its basic properties. In Section 4, we obtain the molecular decomposition for $H^p(\mathbb{Z}^n)$, 
$\frac{n-1}{n} < p \leq 1$, and give a criterion for the $H^p(\mathbb{Z}^n) \to H^q(\mathbb{Z}^n)$ boundedness of certain linear operators, when $\frac{n-1}{n} < p \leq q \leq 1$. Finally, in Section 5 we prove the $H^{p}(\mathbb{Z}^n) \to H^{q}(\mathbb{Z}^n)$ boundedness for the discrete Riesz potential.

Throughout this paper, $C$ will denote a positive real constant not necessarily the same at each occurrence.

\section{Preliminaries} 

The following results will be useful in the study of the discrete molecules presented in Section 3 and in the obtaining of the 
$H^p(\mathbb{Z}^n) \to H^q(\mathbb{Z}^n)$ boundedness for the discrete Riesz potential which will be established in Section 5.

In the sequel, for $j =(j_1, ..., j_n) \in \mathbb{Z}^{n}$ we put $|j|_{\infty} = \max \{ |j_k| : k=1, ..., n \}$ and 
$|j| = (j_1^2 + \cdot \cdot \cdot + j_n^2 )^{1/2}$.

\begin{lemma} \label{series0} If $\epsilon > 0$ and $N \in \mathbb{N}$, then 
\begin{equation}
\sum_{|j|_{\infty} \geq N} \frac{1}{|j|^{n+\epsilon}} \leq 2^{n}n^{n+\epsilon} \left(2+\frac{2^{\frac{\epsilon}{n}} n}{\epsilon} \right)^{n} 
N^{-\epsilon}.
\end{equation}
\end{lemma}

\begin{proof} By Multinomial Theorem, we obtain that
\begin{equation} \label{multinomial ineq}
(|j_1| + |j_2| + \cdot \cdot \cdot + |j_n|)^{n + \epsilon} \geq \max \{N^{1+ \frac{\epsilon}{n}},|j_1|^{1+ \frac{\epsilon}{n}}\} \cdot 
\max \{N^{1+ \frac{\epsilon}{n}},|j_2|^{1+ \frac{\epsilon}{n}}\}
\end{equation}
\[
\cdot \cdot \cdot \max \{ N^{1+ \frac{\epsilon}{n}}, |j_n|^{1+ \frac{\epsilon}{n}} \},
\]
for all $j =(j_1, ..., j_n) \in \mathbb{Z}^{n}$ such that $|j|_{\infty} \geq N$.

On the other hand
\begin{equation} \label{estimate norm}
|j_1| + |j_2| + \cdot \cdot \cdot + |j_n| \leq n  (j_1^{2} + j_2^2 + \cdot \cdot \cdot + j_n^{2})^{1/2},
\end{equation}
for all $(j_1, ..., j_n) \in \mathbb{Z}^n$.

Now, by (\ref{estimate norm}) and (\ref{multinomial ineq}), for every $L \geq N$ we have
\[
\sum_{N \leq | j |_{\infty} \leq L} \frac{1}{|j|^{n+\epsilon}} \leq 2^{n} \sum_{j_n =0}^{L} \cdot \cdot \cdot 
\sum_{j_2 =0}^{L} \sum_{j_1 = N}^{L} \frac{1}{(j_1^{2} + j_2^2 + \cdot \cdot \cdot + 
j_n^{2})^{\frac{n + \epsilon}{2}}}.
\]
\[
\leq 2^{n} \sum_{j_n =0}^{L} \cdot \cdot \cdot \sum_{j_2 =0}^{L} 
\sum_{j_1 = N}^{L} \frac{n^{n + \epsilon}}{ \max \{N^{1+ \frac{\epsilon}{n}},|j_1|^{1+ \frac{\epsilon}{n}}\} \cdot 
\max \{N^{1+ \frac{\epsilon}{n}},|j_2|^{1+ \frac{\epsilon}{n}}\} 
\cdot \cdot \cdot \max \{ N^{1+ \frac{\epsilon}{n}}, |j_n|^{1+ \frac{\epsilon}{n}} \}}
\]
\[
\leq 2^{n}n^{n + \epsilon} \left( N^{-\frac{\epsilon}{n}} + \sum_{j_1=N}^{L} \frac{1}{j_1^{1 + \frac{\epsilon}{n}}} \right)^n.
\]
Finally, letting $L$ tend to infinity, we obtain
\[
\sum_{|j|_{\infty} \geq N} \frac{1}{|j|^{n+\epsilon}} \leq 2^{n}n^{n+\epsilon} \left(2+\frac{2^{\frac{\epsilon}{n}}n}{\epsilon} \right)^{n} 
N^{-\epsilon}.
\]
\end{proof}

\begin{lemma} (\cite[Example 3.2.10]{Grafakos}) \label{estim local}
Let $0 < \alpha < n$ and $\Phi \in \mathcal{S}(\mathbb{R}^n)$ such that $\widehat{\Phi}(\xi) = 1$ if $|\xi| \leq 1$ and
$\widehat{\Phi}(\xi) = 0$ if $|\xi| > 2$. For $x \in \mathbb{R}^n$ and $R > 0$, we put
\[
\mu_{\alpha, R}(x) = \displaystyle{\sum_{j \in \mathbb{Z}^n \setminus \{{\bf 0}\}}} |j|^{\alpha - n} \, \widehat{\Phi}(j/R) \, 
e^{2\pi i(j \cdot x)}.
\] 
Then,
\[
\lim_{R \to \infty}| \mu_{\alpha, R}(x) | \leq C |x|^{-\alpha}, \,\,\,\,\, \text{for all} \,\,\, 
x \in [-1/2, 1/2)^n \setminus \{ {\bf 0} \},
\]
where $C$ is independent of $x$.
\end{lemma}

\begin{lemma} \label{FT atom}
Let $0 < p \leq 1 < p_0 < \infty$ and $L \geq d_p$. Given an $(p, p_0, L)$-atom $a = \{ a(j) \}_{j \in \mathbb{Z}^n}$, we put
\[
\widehat{a}(x) = \sum_{j \in \mathbb{Z}^n} a(j) e^{2\pi i (j \cdot x)}, \,\,\,\, x \in \mathbb{R}^n.
\]
Then
\[
|\widehat{a}(x)| \leq (2\pi)^{L+1} |x |^{L+1} \sum_{j \in \mathbb{Z}^n} |a(j)| |j|^{L+1} e^{2\pi\sqrt{n}|j|}, \,\,\,\,\, \text{for} \,\, 
|x| \leq \sqrt{n}.
\]
\end{lemma}

\begin{proof} From the identity $e^{2\pi i (j\cdot x)} = \sum_{k=0}^{\infty} \frac{(2\pi i (j\cdot x))^k}{k!}$, we have that
\[
e^{2\pi i (j\cdot x)} - \sum_{k=0}^{L} \frac{(2\pi i (j\cdot x))^k}{k!} = \sum_{k=L+1}^{\infty} \frac{(2\pi i (j\cdot x))^k}{k!}.
\]
So,
\begin{equation} \label{exp Taylor}
\left|e^{2\pi i (j\cdot x)} - \sum_{k=0}^{L} \frac{(2\pi i (j\cdot x))^k}{k!}\right| \leq (2\pi)^{L+1}|j|^{L+1}|x|^{L+1} e^{2\pi|j||x|}.
\end{equation}
Now, by the moment condition of the atom $a(\cdot)$, we have that
\begin{equation} \label{FT atom2}
\widehat{a}(x) = \sum_{j \in \mathbb{Z}^n} a(j) \left( e^{2\pi i (j \cdot x)} - \sum_{k=0}^{L} \frac{(2\pi i (j\cdot x))^k}{k!}  \right).
\end{equation}
Finally, (\ref{FT atom2}) and (\ref{exp Taylor}) lead to
\[
|\widehat{a}(x)| \leq  (2\pi)^{L+1} |x|^{L+1} \sum_{j \in \mathbb{Z}^n} |a(j)| |j|^{L+1} e^{2\pi\sqrt{n}|j|},
\]
for all $|x| \leq \sqrt{n}$.
\end{proof}

Given $R>0$, we consider the set $E_R$ of slowly increasing $C^{\infty}(\mathbb{R}^n)$ functions $f$ with 
$\supp \widehat{f} \subset [-R, R]^n$. The elements of $E_R$ are functions of exponential type $R$. For this class of functions we have the following result.

\begin{lemma} (\cite[Lemma 3]{Auscher}) \label{exp type}
Let $0 < p < \infty$ and $0 < R < \frac{1}{2}$, then there exists a positive constant $C$ such that
\[
\sum_{j \in \mathbb{Z}^n} |f(j)|^p \leq C \int_{\mathbb{R}^n} |f(x)|^p dx,
\]
for every function $f$ of exponential type $R$.
\end{lemma}

\section{Molecules in $\mathbb{Z}^n$}

We introduce the concept of molecule in $\mathbb{Z}^n$ and will prove that the discrete molecules belong to $H^p(\mathbb{Z}^n)$.

\begin{definition}
Let $0 < p \leq 1 < p_0 \leq \infty$, $r > \frac{1}{p} - \frac{1}{p_0}$ and $d_p = \lfloor n (p^{-1} - 1) \rfloor$. A sequence 
$M = \{ M(j) \}_{j \in \mathbb{Z}^n}$ is called a $(p, p_0, r, d_p)$-molecule centered at $m_0 \in \mathbb{Z}^n$ if $M$ satisfies the following conditions:

(m1) $\mathcal{N}_{p, p_0, r}(M) := \| M \|_{\ell^{p_0}}^{1-\theta} \| | \cdot - m_0|^{n r} M \|_{\ell^{p_0}}^{\theta} < \infty$, where 
$\theta = (1/p - 1/p_0)/r$, 

(m2) $\displaystyle{\sum_{j \in \mathbb{Z}^n}} j^{\beta} M(j) =0$, for all multi-index $\beta=(\beta_1, ..., \beta_n) \in \mathbb{N}_0^n$ with $[\beta] \leq d_p$. 
\end{definition}
We call $\mathcal{N}_{p, p_0, r}(M)$ the molecule norm of $M$, which is also denoted by $\mathcal{N}(M)$.

\begin{remark} \label{atom is mol}
Let $0 < p \leq 1 < p_0 \leq \infty$, $L \geq d_p$ and $r > \frac{1}{p} - \frac{1}{p_0}$. If $a(\cdot)$ is an $(p, p_0, L)$-atom centered at a cube $Q \subset \mathbb{Z}^n$, then $a(\cdot)$ is a $(p, p_0, r, d_p)$-molecule centered at each $m_0 \in Q$ with $\mathcal{N}(a) \leq C$, where $C$ is independent of the atom $a(\cdot)$ and $m_0$.
\end{remark}

\begin{lemma} \label{M ellp}
Let $0 < p \leq 1 < p_0 \leq \infty$, $r > \frac{1}{p} - \frac{1}{p_0}$. If $M$ is a $(p, p_0, r, d_p)$-molecule centered at 
$m_0 \in \mathbb{Z}^n$, then
\[
\| M \|_{\ell^p(\mathbb{Z}^n)} \leq C \mathcal{N}(M),
\]
where $C$ depends only on $n$, $p$, $p_0$ and $r$.
\end{lemma}

\begin{proof} Without loss of generality we assume $m_0 = {\bf 0}$. From the moment condition (m2) of the molecule $M$, it follows that
\[
|M({\bf 0})| \leq  \sum_{j \neq {\bf 0}} |M(j)|.
\]
Then, for $0 < p \leq 1$
\[
\sum_{j \in \mathbb{Z}^n} |M(j)|^p \leq 2 \sum_{j \neq {\bf 0}} |M(j)|^p.
\]
By H\"older inequality, we have
\[
\sum_{j \neq {\bf 0}} |M(j)|^p \leq \left(\sum_{j \neq {\bf 0}} |j|^{-nrp(p_0/p)'} \right)^{1/(p_0/p)'}
\left( \sum_{j \neq {\bf 0}} |j|^{nrp_0}|M(j)|^{p_0} \right)^{p/p_0}.
\]
We observe that $rp(p_0/p)' -1 > 0$. Then, by applying Lemma \ref{series0} with $\epsilon = n\left(rp(p_0/p)' -1 \right)$ and $N=1$, we obtain
\begin{equation} \label{pp0 estimate}
\| M \|_{\ell^p(\mathbb{Z}^n)} \leq C \| |\cdot |^{nr} M  \|_{\ell^{p_0}(\mathbb{Z}^n)}.
\end{equation}
We put $\sigma = \left( \|M \|_{\ell^{p_0}}^{-1} \| |\cdot |^{nr} M  \|_{\ell^{p_0}} \right)^{1/nr}$. If $\sigma \leq 2$, then 
$\| |\cdot |^{nr} M  \|_{\ell^{p_0}} \leq 2^{nr} \|M \|_{\ell^{p_0}}$. This and (\ref{pp0 estimate}) lead to
\[
\| M \|_{\ell^p} \leq C \|M \|_{\ell^{p_0}}^{1-\theta} \| |\cdot |^{nr} M  \|_{\ell^{p_0}}^{\theta} = C \mathcal{N}(M),
\]
where $\theta = (1/p - 1/p_0)/r$. Now, we assume $\sigma > 2$, and put
\[
\sum_{j \in \mathbb{Z}^n} |M(j)|^p = \sum_{|j|_{\infty} < \lfloor \sigma \rfloor} |M(j)|^p + \sum_{|j|_{\infty} \geq \lfloor \sigma \rfloor} |M(j)|^p = J_1 + J_2.
\]
For $J_1$, by H\"older inequality, we have
\[
J_1 \leq \| M \|_{\ell^{p_0}}^p \left( \sum_{|j|_{\infty} < \lfloor \sigma \rfloor} 1  \right)^{1/(p_0/p)'} \leq C \| M \|_{\ell^{p_0}}^p
\sigma ^{n \left( 1-\frac{p}{p_0} \right)} = C \mathcal{N}(M)^p.
\]
For $J_2$, by H\"older inequality and Lemma \ref{series0} with $\epsilon = n\left(rp(p_0/p)' -1 \right)$ and $N=\lfloor \sigma \rfloor$, 
we obtain
\[
J_2 \leq \| |\cdot |^{nr} M  \|_{\ell^{p_0}}^p \left(\sum_{|j|_{\infty} \geq \lfloor \sigma \rfloor} |j|^{-nrp(p_0/p)'} \right)^{1/(p_0/p)'}
\]
\[
\leq C \| |\cdot |^{nr} M  \|_{\ell^{p_0}}^p \lfloor \sigma \rfloor^{-n \left(rp - \frac{1}{(p_0/p)'} \right)},
\]
since $\sigma > 2$, it follows that
\[
\leq C \| |\cdot |^{nr} M  \|_{\ell^{p_0}}^p (\sigma -1)^{-n \left(rp - \frac{1}{(p_0/p)'} \right)}
\]
\[
\leq C \| |\cdot |^{nr} M  \|_{\ell^{p_0}}^p \sigma^{-n \left(rp - \frac{1}{(p_0/p)'} \right)} = C \mathcal{N}(M)^p.
\]
This concludes the proof.
\end{proof}

\begin{lemma} \label{Mol cont}
Let $0 < p \leq 1 < p_0 \leq \infty$, $r > \frac{1}{p} - \frac{1}{p_0}$ and let $M$ be a $(p, p_0, r, d_p)$-molecule centered at 
$m_0 \in \mathbb{Z}^n$. For $\Phi \in \mathcal{S}(\mathbb{R}^n)$, we put
\[
\widetilde{M}(x) = \sum_{j \in \mathbb{Z}^n} M(j) \Phi(x-j), \,\,\,\, x \in \mathbb{R}^n.
\]
Then, $\widetilde{M}^d$ is a $(p, p_0, r, d_p)$-molecule centered at $m_0$ with
\[
\mathcal{N}(\widetilde{M}^d) \leq C \mathcal{N}(M),
\]
where $C$ does not depends on $M$, and $\widetilde{M}^d$ is the restriction of $\widetilde{M}$ to $\mathbb{Z}^n$.
\end{lemma}

\begin{proof} The proof is similar to the one given in \cite[Lemma 2]{Kanjin}.
\end{proof}

\begin{proposition} \label{norm mol}
Let $n \geq 2$, $\frac{n-1}{n} < p \leq 1 < p_0 \leq \infty$, and $r > \frac{1}{p} - \frac{1}{p_0}$. If $M$ is a $(p, p_0, r, 0)$-molecule, 
then $M \in H^p(\mathbb{Z}^n)$ with 
\[
\| M \|_{H^p(\mathbb{Z}^n)} \leq C \mathcal{N}(M),
\] 
where $C$ is independent of $M$.
\end{proposition}

\begin{proof}
For $n \geq 2$, $\frac{n-1}{n} < p \leq 1$ we have $d_p = 0$, by Lemma \ref{M ellp}, (\ref{Riesz norm}) and \cite[Theorem 2.6]{Carro}, it is enough to show 
\begin{equation} \label{norma Riesz}
\sum_{s=1}^{n}\| R^d_s M \|_{\ell^p(\mathbb{Z}^n)} \leq C \mathcal{N}(M),
\end{equation}
for discrete $(p, p_0, r, 0)$-molecules $M$. From the invariance by translations of the operator $R^d_s$, we can assume that the molecules are centered at ${\bf 0}$. Then, we fix a real number $0 < R < \frac{1}{2}$ and let $\Phi$ be a radial function of $\mathcal{S}(\mathbb{R}^n)$ such that $\supp \widehat{\Phi} \subset [-R, R]^n$ with $\widehat{\Phi}(\xi) \equiv 1$ on the ball $B({\bf 0}, R/2)$. Given a discrete $(p, p_0, r, 0)$-molecule $M$, by Lemma \ref{M ellp}, we have that 
$M \in \ell^p{(\mathbb{Z}^n)} \subset \ell^1{(\mathbb{Z}^n)}$, and so
\[
\widetilde{M}(x) = \sum_{j \in \mathbb{Z}^n} M(j) \Phi(x-j) \in L^p(\mathbb{R}^n) \cap L^1(\mathbb{R}^n) \cap E_R \subset L^2(\mathbb{R}^n).
\]
For $s=1, ..., n$ and $x \in \mathbb{R}^n$, using Fourier’s inversion theorem, we introduce the continuous Riesz transforms $R_s$ acting on 
$\widetilde{M}$
\begin{eqnarray*}
(R_s \widetilde{M})(x) &=& \left( \xi \to -i C_n \frac{\xi_s}{|\xi|} \widehat{\widetilde{M}}(\xi) \right)^{\bigvee}(x) \\
&=& \int_{\mathbb{R}^n} (-i) C_n \frac{\xi_s}{|\xi|} \sum_{j \in \mathbb{Z}^n} M(j) \, \widehat{\Phi}(\xi)  \, e^{2\pi i(x-j) \cdot \xi} \\
&=& \sum_{j \in \mathbb{Z}^n} M(j) (R_s \Phi)(x - j).
\end{eqnarray*}
Proceeding as in the proof of Theorem 2.6 in \cite{Carro}, and by Lemma \ref{M ellp}, we obtain
\begin{equation} \label{norma Riesz2}
\left\| (R_s \widetilde{M})^d - R_s^d M \right\|_{\ell^p(\mathbb{Z}^n)} \leq C \| M \|_{\ell^p(\mathbb{Z}^n)} \leq C \mathcal{N}(M),
\end{equation}
where $(R_s \widetilde{M})^d$ is the restriction of $R_s \widetilde{M}$ to $\mathbb{Z}^n$. 

On the other hand, Lemma \ref{exp type} applied to the function $R_s \widetilde{M}$ of exponential type $R$ allows us to obtain
\begin{equation} \label{ellpLp}
\| (R_s \widetilde{M})^d \|_{\ell^p(\mathbb{Z}^n)} \leq C \| R_s \widetilde{M} \|_{L^p(\mathbb{R}^n)}.
\end{equation}
Finally, according to the ideas to estimate $(9)$ in \cite[Proposition 1]{Kanjin} and Lemma \ref{Mol cont}, we obtain that $\widetilde{M}$ 
is a $p$-molecule centered at ${\bf 0}$ in $\mathbb{R}^n$ with 
\begin{equation} \label{normas N}
\widetilde{\mathcal{N}}(\widetilde{M}) \leq C \mathcal{N}(\widetilde{M}^d) \leq C \mathcal{N}(M),
\end{equation}
where $C$ is independent of $M$ and  $\widetilde{\mathcal{N}}(\widetilde{M})$ is the continuous molecule norm. Thus
$\widetilde{M} \in H^p(\mathbb{R}^n)$. Since $R_s$ is a bounded operator $H^p(\mathbb{R}^n) \to L^p(\mathbb{R}^n)$ for $0 < p \leq 1$, 
then (\ref{ellpLp}) and (\ref{normas N}) lead to
\[
\| (R_s \widetilde{M})^d \|_{\ell^p(\mathbb{Z}^n)} \leq C \| R_s \widetilde{M} \|_{L^p(\mathbb{R}^n)} \leq 
C \| \widetilde{M} \|_{H^p(\mathbb{R}^n)} \leq 
\widetilde{\mathcal{N}}(\widetilde{M}) \leq C \mathcal{N}(M).
\]
So, this inequality and (\ref{norma Riesz2}) give (\ref{norma Riesz}). Therefore the proof is concluded.
\end{proof}

\section{Molecular decomposition for $H^p(\mathbb{Z}^n)$}

In this section, we establish a molecular reconstruction theorem for $H^p(\mathbb{Z}^n)$, $\frac{n-1}{n} < p \leq 1$. As an application of this result, we give a criterion for the $H^p(\mathbb{Z}^n) \to H^q(\mathbb{Z}^n)$ boundedness of certain linear operators, when 
$\frac{n-1}{n} < p \leq q \leq 1$.

\begin{theorem} \label{mol decomp}
Let $n \geq 2$, $\frac{n-1}{n} < p \leq 1 < p_0 \leq \infty$, $r > \frac{1}{p} - \frac{1}{p_0}$. 
If $\{ M_k \}_{k=1}^{\infty}$ is a sequence of $(p, p_0, r, 0)$-molecules in $\mathbb{Z}^n$ such that 
$\sum_{k=1}^{\infty} \mathcal{N}(M_k)^p < \infty$, then the series $\sum_{k=1}^{\infty} M_k$ converges to a sequence
$h$ in $H^p(\mathbb{Z}^n)$ and
\begin{equation}  \label{h}
\| h \|_{H^p(\mathbb{Z}^n)}^p \leq C \sum_{k=1}^{\infty} \mathcal{N}(M_k)^p,
\end{equation}
where $C$ does not depend on the molecules $M_k$.
\end{theorem}

\begin{proof}
For each $L \in \mathbb{N}$, we consider $S_{L} = \sum_{k=1}^{L} M_k$. By proposition \ref{norm mol} we have that
$\| S_{\widetilde{L}} - S_L \|_{H^p(\mathbb{Z}^n)}^p \leq \sum_{k=L}^{\widetilde{L}} \mathcal{N}(M_k)^p$ for every $L \leq \widetilde{L}$. Since $\sum_{k=1}^{\infty} \mathcal{N}(M_k)^p < \infty$, from the completeness of $H^p(\mathbb{Z}^n)$ with respect to the metric 
$d(f,g) = \| f-g \|_{H^p(\mathbb{Z}^n)}^p$, it follows that there exists $h \in H^p(\mathbb{Z}^n)$ such that $h=\sum_{k=1}^{\infty} M_k$ in 
$H^p(\mathbb{Z}^n)$. Finally, the inequality 
$\| h \|_{H^p(\mathbb{Z}^n)}^p \leq \| h - S_L \|_{H^p(\mathbb{Z}^n)}^p + \| S_L \|_{H^p(\mathbb{Z}^n)}^p$ gives (\ref{h}).
\end{proof}

\begin{remark}
By Remarks \ref{infinity atom} and \ref{atom is mol}, we have that every $(p, \infty, L)$-atom is a $(p, p_0, r, 0)$-molecule. Then, Theorem \ref{atom decomp} and Theorem \ref{mol decomp} give the molecular characterization for $H^p(\mathbb{Z}^n)$, $\frac{n-1}{n} < p \leq 1$.
\end{remark}

\begin{corollary} \label{Criterion}
Let $n \geq 2$, $\frac{n-1}{n} < p \leq q \leq 1 < p_0 \leq q_0 \leq \infty$, $r > \frac{1}{q} - \frac{1}{q_0}$, $L \geq 0$ and let $T$ be a bounded linear operator $\ell^{p_0}(\mathbb{Z}^n) \to \ell^{q_0}(\mathbb{Z}^n)$ such that $Ta(\cdot)$ is a $(q, q_0, r, 0)$-molecule for all 
$(p, \infty, L)$-atom $a(\cdot)$. Suppose that there exists an universal positive constant $C_0$ such that 
$\mathcal{N}(Ta) \leq C_0$ for all $(p, \infty, L)$-atom $a(\cdot)$, then there exists a positive constant $C$ such that
\begin{equation} \label{acot HpHq}
\| Tb \|_{H^q(\mathbb{Z}^n)} \leq C \| b \|_{H^p(\mathbb{Z}^n)},
\end{equation}
for all $b \in H^p(\mathbb{Z}^n)$.
\end{corollary}

\begin{proof}
For $n \geq 2$ and $\frac{n-1}{n} < p \leq q \leq 1$, we have that $d_p = d_q = 0$. Given $b \in H^p(\mathbb{Z}^n)$ and $L \geq 0$, by Theorem \ref{atom decomp}, there exist a sequence of real numbers $\{ \lambda_k \}_{k=1}^{\infty}$, a sequence of $(p, \infty, L)$-atoms 
$\{ a_k \}_{k=1}^{\infty}$ such that $b = \sum_{k=1}^{\infty} \lambda_k a_k$ converges in $\ell^{p}(\mathbb{Z}^n)$ and 
\begin{equation} \label{norm Hp}
\left(\sum_{k=1}^{\infty} |\lambda_k|^p \right)^{1/p} \leq C \| b \|_{H^p(\mathbb{Z}^n)}.
\end{equation}
Since $\ell^{p}(\mathbb{Z}^n) \subset \ell^{p_0}(\mathbb{Z}^n)$ embed continuously, we have that $b = \sum_k \lambda_k a_k$ in 
$\ell^{p_0}(\mathbb{Z}^n)$. Being $T$ a bounded operator $\ell^{p_0}(\mathbb{Z}^n) \to \ell^{q_0}(\mathbb{Z}^n)$, we obtain that 
\begin{equation} \label{Tb}
(Tb)(j) = \sum_{k=1}^{\infty} \lambda_k (T a_k)(j), \,\,\,\, \text{for all} \,\, j \in \mathbb{Z}^n.
\end{equation}
By our hypotheses, $M_k = \lambda_k (Ta_k)$ is a $(q, q_0, r, 0)$-molecule for every $k$ and
\begin{equation} \label{Nlambda}
\mathcal{N}(M_k) \leq C_0 |\lambda_k|.
\end{equation}
Now, by (\ref{Nlambda}), (\ref{norm Hp}) and Theorem \ref{mol decomp}, there exists $h \in H^p(\mathbb{Z}^n)$ such that 
$h(j) = \sum_{k=1}^{\infty} M_k(j) = \sum_{k=1}^{\infty} \lambda_k (T a_k)(j)$ for all $j \in \mathbb{Z}^n$. So, (\ref{Tb}) gives $h = Tb$. Then, (\ref{h}), (\ref{Nlambda}) and (\ref{norm Hp}) allow us to get (\ref{acot HpHq}).
\end{proof}

\section{Discrete Riesz potential}

Let $0 < \alpha < n$, $0 < p < \frac{n}{\alpha}$ and $\frac{1}{q} = \frac{1}{p} - \frac{\alpha}{n}$. In this section we obtain the
$H^p(\mathbb{Z}^n) \to H^q(\mathbb{Z}^n)$ boundedness of discrete Riesz potential $I_{\alpha}$, which is defined by
\begin{equation} \label{discrete Riesz pot}
(I_{\alpha}b)(j) = \sum_{i \in \mathbb{Z}^n \setminus \{ j \}} \frac{b(i)}{|i-j |^{n - \alpha}}, \,\,\,\,\,\, j \in \mathbb{Z}^n.
\end{equation}

We first prove that the operator $I_{\alpha}$ maps atoms into molecules.

\begin{proposition} \label{Ta mol}
Let $n \geq 2$, $0 < \alpha < \frac{n}{n-1}$, $\frac{n-1}{n} < p \leq \frac{n}{n+\alpha}$, $\frac{1}{q} = \frac{1}{p} - \frac{\alpha}{n}$, 
$\frac{n}{n-\alpha} < q_0 < \infty$, $r > \frac{1}{q} - \frac{1}{q_0}$, $L = \lfloor n(r + \frac{1}{q_0}) + \alpha \rfloor +1$ and let 
$I_{\alpha}$ be the discrete Riesz potential given by (\ref{discrete Riesz pot}). Then $I_{\alpha}a(\cdot)$ is a $(q, q_0, r, 0)$-molecule for every $(p, \infty, L)$-atom $a(\cdot)$. Moreover, there exists an universal positive constat $C_0$ such that 
$\mathcal{N}(I_{\alpha}a) \leq C_0$ for all $(p, \infty, L)$-atom $a(\cdot)$.
\end{proposition}

\begin{proof}
Given an $(p, \infty, L)$-atom $a(\cdot)$ centered at $m_0$, we shall prove that $I_{\alpha}a(\cdot)$ is a $(q, q_0, r, 0)$-molecule centered at $m_0$. For them, let 
$Q = \{ j \in \mathbb{Z}^n : |j-m_0|_{\infty} \leq N \}$ be the discrete cube on which $a(\cdot)$ is supported. Now, we put
$Q^{*} = \left\{ j \in \mathbb{Z}^n : |j-m_0|_{\infty} \leq 4 \lfloor \sqrt{n} \rfloor N \right\}$ and write 
$U_1 = \left( \sum_{j \in Q^{*}} \left|(I_{\alpha} a)(j) \, |j - m_0|^{nr}\right|^{q_0} \right)^{1/q_0}$ and
$U_2 = \left( \sum_{j \notin Q^{*}} \left|(I_{\alpha} a)(j) \, |j - m_0|^{nr}\right|^{q_0} \right)^{1/q_0}$, where
$\frac{n}{n-\alpha} < q_0 < \infty$ is fixed. Then, we define $\frac{1}{p_0} := \frac{1}{q_0} + \frac{\alpha}{n}$. By 
Remark \ref{infinity atom}, we have that $a(\cdot)$ is an $(p, p_0, L)$-atom.
Now, \cite[Proposition (a)]{Wainger} gives 
\[
U_1 \leq C \| I_{\alpha}a \|_{\ell^{q_0}} N^{nr} \leq C \|a \|_{\ell^{p_0}} N^{nr}.
\] 
Following the ideas to obtain $(14)$ in \cite[Proof of Theorem 3.3]{Rocha2}, we obtain
\[
U_2 \leq C N^{L + 1} \sum_{k \in Q} |a(k)|
\left(\sum_{j \notin Q^{*}} |j-m_0|^{(nr + \alpha - n - L - 1) q_0}\right)^{1/q_0}. 
\]
From H\"older's inequality and Lemma \ref{series0} with $\epsilon = -n +(n + L + 1 - \alpha - nr)q_0 > 0$, we get 
\[
U_2 \leq C \|a \|_{\ell^{p_0}} N^{L + 1 + n - n/p_0} N^{nr + \alpha - n - L - 1 + n/q_0}
= C \|a \|_{\ell^{p_0}} N^{nr}.
\]
Thus, we have
\[
\| I_{\alpha}a(\cdot) | \cdot - m_0|^{nr} \|_{\ell^{q_0}} \leq C (U_1 + U_2) \leq C \|a \|_{\ell^{p_0}} N^{nr}.
\]
Then, for $\theta = (1/q - 1/q_0)/r = (1/p - 1/p_0)/r$, it follows that
\[
\mathcal{N}(I_{\alpha}a) \leq C \|a \|_{\ell^{p_0}}^{1-\theta} \|a \|_{\ell^{p_0}}^{\theta} N^{nr\theta} 
= C \|a \|_{\ell^{p_0}} N^{n(1/p - 1/p_0)} \leq C_0,
\]
where $C_0$ does not depend on the atom $a(\cdot)$. This gives the size condition (m1).

For $R > 0$ and $\Phi$ as in Lemma \ref{estim local}, we put
\[
(I_{\alpha, R} \, a)(j) = \sum_{k \neq {\bf 0}} |k|^{\alpha-n} \, \widehat{\Phi}(k/R) \, a(j-k),
\]
and
\[
\widehat{(I_{\alpha, R} \, a)}(x) = \sum_{j \in \mathbb{Z}^n} (I_{\alpha, R} \, a)(j) \, e^{2\pi i (j \cdot x)}.
\]
A computation gives
\[
\widehat{(I_{\alpha, R} \, a)}(x) = \mu_{\alpha, R}(x) \, \widehat{a}(x).
\]
By Lemma \ref{estim local} and Lemma \ref{FT atom}, we have
\begin{equation} \label{moment for Ia}
|\widehat{(I_{\alpha} \, a)}(x)| = \lim_{R \to \infty} |\widehat{(I_{\alpha, R} \, a)}(x)| = \lim_{R \to \infty} |\mu_{\alpha, R}(x)| \, |\widehat{a}(x)|
\end{equation}
\[
\leq C |x|^{-\alpha} |x|^{L + 1} \sum_{j \in \mathbb{Z}^n} |a(j)| |j|^{L + 1} 
e^{2\pi \sqrt{n} |j|},
\]
for all $x \in [-1/2, 1/2)^n \setminus \{ {\bf 0 }\}$. Since $L + 1 - \alpha > 0$, after letting $x \to {\bf 0}$ in 
(\ref{moment for Ia}), we obtain
\[
\left| \sum_{j \in \mathbb{Z}^n} (I_{\alpha}a)(j) \right|=|\widehat{(I_{\alpha} \, a)}({\bf 0})| = 0,
\]
which is the moment condition (m2) for the range $\frac{n-1}{n} < p \leq 1$.
\end{proof}

\begin{theorem}
Let $0 < \alpha < n$ and let $I_{\alpha}$ be the discrete Riesz potential given by (\ref{discrete Riesz pot}). If $\frac{n-1}{n} < p < \frac{n}{\alpha}$ and $\frac{1}{q} = \frac{1}{p} - \frac{\alpha}{n}$, then there exists a positive constant $C$ such that
\begin{equation} \label{main result}
\| I_{\alpha} b \|_{H^q(\mathbb{Z}^n)} \leq C \| b \|_{H^p(\mathbb{Z}^n)}
\end{equation}
for $b \in H^p(\mathbb{Z}^n)$. 
\end{theorem}

\begin{proof} The case $n=1$ was proved in \cite[Theorem 4]{Kanjin}. From now on, we consider $n \geq 2$.
For the case $1 < p < \frac{n}{\alpha}$ and $\frac{1}{q} = \frac{1}{p} - \frac{\alpha}{n}$ (i.e.: $1 < p < q < \infty$), we have 
that $H^p(\mathbb{Z}^n) = \ell^p(\mathbb{Z}^n)$ and $H^q(\mathbb{Z}^n) = \ell^q(\mathbb{Z}^n)$. So, the $\ell^p(\mathbb{Z}^n) \to \ell^q(\mathbb{Z}^n)$ boundedness of $I_{\alpha}$ follows from 
\cite[Proposition (a)]{Wainger} or \cite[Theorem 3.1]{Rocha2}.

For $\frac{n}{n + \alpha} < p \leq 1$ and $\frac{1}{q} = \frac{1}{p} - \frac{\alpha}{n}$ (i.e.: $1 < q < \infty$), we have that 
$H^p(\mathbb{Z}^n) \subsetneqq \ell^p(\mathbb{Z}^n)$ and $H^q(\mathbb{Z}^n) = \ell^q(\mathbb{Z}^n)$. In this case, 
\cite[Theorem 3.3]{Rocha2} gives the $H^p(\mathbb{Z}^n) \to \ell^q(\mathbb{Z}^n)$ boundedness of $I_{\alpha}$.

If $ \frac{n}{n-1} \leq \alpha <n$, we have that $\frac{n}{n + \alpha} \leq \frac{n-1}{n}$. This is contemplated in the previous case. Thus, 
for $0 < \alpha < \frac{n}{n-1}$, $\frac{n-1}{n} < p \leq \frac{n}{n + \alpha}$ and $\frac{1}{q} = \frac{1}{p} - \frac{\alpha}{n}$ 
(i.e.: $\frac{n-1}{n} < p < q \leq 1$), we have that $H^p(\mathbb{Z}^n) \subsetneqq \ell^p(\mathbb{Z}^n)$ and 
$H^q(\mathbb{Z}^n) \subsetneqq \ell^q(\mathbb{Z}^n)$. Here is where we apply the molecular reconstruction to establish the 
$H^p(\mathbb{Z}^n) \to H^q(\mathbb{Z}^n)$ boundedness of $I_{\alpha}$. For them, we take $\frac{n}{n-\alpha} < q_0 < \infty$, 
$r > \frac{1}{q} - \frac{1}{q_0}$ and put $\frac{1}{p_0} := \frac{1}{q_0} + \frac{\alpha}{n}$, since $I_{\alpha}$ is a bounded linear operator $\ell^{p_0}(\mathbb{Z}^n) \to \ell^{q_0}(\mathbb{Z}^n)$, Proposition \ref{Ta mol} and Corollary \ref{Criterion} 
(with $L= \lfloor n(r + \frac{1}{q_0}) + \alpha \rfloor +1$) allow us to obtain (\ref{main result}). This finishes the proof.
\end{proof}

\bigskip
\address{
Departamento de Matem\'atica \\
Universidad Nacional del Sur (UNS) \\
Bah\'{\i}a Blanca, Argentina}
{pablo.rocha@uns.edu.ar}

\end{document}